\documentclass[a4paper, 12pt]{amsart}
\usepackage{fourier}
\usepackage{amssymb}
\usepackage{amsmath}
\usepackage{amscd}
\usepackage{amsthm}
\usepackage[centertags]{amsmath}
\usepackage{amsfonts}
\usepackage{newlfont}
\usepackage[all]{xy}
\usepackage{graphicx}
\usepackage{amsfonts, amssymb}
\usepackage[usenames]{color}
\usepackage{mathrsfs}
\usepackage{latexsym}

\newtheorem{thm}{Theorem}[section]
\newtheorem{cor}[thm]{Corollary}
\newtheorem{lem}[thm]{Lemma}
\newtheorem{prop}[thm]{Proposition}

\theoremstyle{definition}

\newtheorem{rem}[thm]{Remark}

\newcommand{\ccc}{\mathbf c}
\newcommand{\ppp}{\mathbf p}
\newcommand{\qqq}{\mathbf q}
\newcommand{\bbb}{\mathbf b}
\newcommand{\mmm}{\mathbf m}
\newcommand{\PP}{\mathcal P}

\newcommand{\zC}{\mathbb C}

\newcommand{\zR}{\mathbb R}
\newcommand{\zN}{\mathbb N}
\newcommand{\zK}{\mathbb K}

\begin{document}

\title[The polarization constant of finite dimensional complex spaces is one]{The polarization constant of finite dimensional complex spaces is one}

\thanks{This work was partially supported by CONICET PIP 11220130100483, CONICET PIP 11220130100329  and ANPCyT PICT 2015-2299.}


\author{Ver\'onica Dimant}
\address{Departamento de Matem\'{a}tica y Ciencias, Universidad de San
Andr\'{e}s, Vito Dumas 284, (B1644BID) Victoria, Buenos Aires,
Argentina and CONICET} \email{vero@udesa.edu.ar}

\author{Daniel Galicer}
\address{Departamento de Matem\'{a}tica, Facultad de Ciencias Exactas y Naturales, Universidad de Buenos Aires, (1428) Buenos Aires,
Argentina and CONICET} \email{dgalicer@dm.uba.ar}

\author[J. T. Rodr\'{i}guez]{Jorge Tom\'as Rodr\'{i}guez}
\address{Departamento de Matem\'{a}tica and NUCOMPA, Facultad de Cs. Exactas, Universidad Nacional del Centro de la Provincia de Buenos Aires, (7000) Tandil, Argentina and CONICET}
\email{jtrodrig@dm.uba.ar}

\begin{abstract} 
The polarization constant of a Banach space $X$ is defined as $$\ccc(X):= \limsup\limits_{k\rightarrow \infty} \ccc(k, X)^\frac{1}{k},$$
where $\ccc(k, X)$ stands for the best constant $C>0$ such that $ \Vert \overset{\vee}{P} \Vert \leq C \Vert P \Vert$ for every $k$-homogeneous polynomial $P \in \mathcal P(^kX)$.
We show that if $X$ is a finite dimensional complex space then $\ccc(X)=1$.  
We derive some consequences of this fact regarding the convergence of analytic functions on such spaces.

The result is no longer true in the real setting. Here we relate this constant with the so-called  Bochnak's complexification procedure.

We also study some other properties connected with polarization. Namely, we provide necessary conditions related with the geometry of $X$  for $c(2,X)=1$ to hold. Additionally we link polarization constants with certain estimates of the nuclear norm of the product of polynomials. 
\end{abstract}

\maketitle

\section{Introduction}


The polarization constants appear naturally  when relating polynomials with multilinear functions. Given a Banach space $X$  over the field $\zK$ (where $\zK$ can  be either the complex numbers $\zC$ or the real numbers $\zR$), a mapping $P:X\to \zK$ is a (continuous) $k$-homogeneous  polynomial if there exists a $k$-linear symmetric mapping ${T:\underbrace{X\times\cdots \times X}_{k\,\,\, times} \to \zK}$ (continuous) such that $P(\mathbf{x})=T(\mathbf{x},\dots,\mathbf{x})$ for all $\mathbf{x}\in X$. By the polarization formula (see for instance \cite[Corollary 1.6]{dineen1999complex}) 
\begin{equation}\label{formula de polarizacion}
T(\mathbf{x}_1, \dots, \mathbf{x}_k) = \frac{1}{k!2^k} \sum_{\varepsilon_i = \pm 1} \varepsilon_1 \dots \varepsilon_k P\left(\sum_{i=1}^k \varepsilon_i \mathbf{x}_i\right), 
\end{equation}
this map is unique and it is written  $\overset{\vee}{P}=T$. The space of continuous $k$-homogeneous polynomials on a Banach space $X$ is denoted by $\mathcal P(^kX)$ and this is a Banach space when endowed with the uniform norm
$$\Vert P \Vert=\sup_{\Vert \mathbf{x} \Vert =1} \vert P(\mathbf{x})\vert.$$

From Equation \eqref{formula de polarizacion} the following polarization inequality easily holds
\begin{equation}\label{ec polarizacion}
\Vert \overset{\vee}{P} \Vert \leq \frac{k^k}{k!} \Vert P \Vert,
\end{equation}
for every $P \in \mathcal P(^kX)$ and all Banach space $X$. The polarization constant $\frac{k^k}{k!}$ is the best possible for the general case. Indeed, if $X=\ell_1$ there is a norm one $k$-homogeneous polynomial $P \in \mathcal P(^k\ell_1)$ such that $\Vert \overset{\vee}{P} \Vert = \frac{k^k}{k!}$ (see for example \cite{H}).  
On the other hand a classical result of Banach \cite{banach1938homogene} asserts that if $\mathcal{H}$ is a Hilbert space then  $\Vert \overset{\vee}{P} \Vert = \Vert P \Vert$, for every $P \in \mathcal P(^k\mathcal{H})$. Therefore it is natural to define \cite[Definition 1.40]{dineen1999complex}, given a fixed Banach space $X$, its so-called \emph{$k$-polarization constant}
\begin{equation} \label{def k-constante}
\ccc(k,X) := \inf\{C>0 : \Vert \overset{\vee}{P} \Vert \leq C \Vert P \Vert, \; \mbox{for all } P \in \mathcal P(^kX) \},
\end{equation}
 and also its \emph{polarization constant} 
\begin{equation} \label{def constante}
\ccc(X):= \limsup\limits_{k\rightarrow \infty} \ccc(k, X)^\frac{1}{k}.
\end{equation}
From inequality \eqref{ec polarizacion} and Stirling's formula we have $1 \leq \ccc(X) \leq e$, where the leftmost value is attained for $X=\ell_2$ and the rightmost value is attained for $X=\ell_1$. The interest of knowing the value of $\ccc(X)$ relies on the fact that it provides accurate  hypercontractive inequalities of the form: 
\begin{equation}
    \Vert \overset{\vee}{P} \Vert \leq C^k \Vert P \Vert, \; \mbox{for all } P \in \mathcal P(^kX) \mbox{ and all $k$ large enough.}
\end{equation}

Our main result shows that the norm of a $k$-homogeneous polynomial over a finite dimensional complex Banach space and the norm of its associated $k$-linear form are quite close, provided $k$ is large enough. Precisely,

\begin{thm}\label{teo fin dim} For any finite dimensional complex Banach space $X$, we have that $$\ccc(X)=1.$$
\end{thm}

As a consequence of Theorem \ref{teo fin dim} we present an application regarding the convergence of analytic functions defined on finite dimensional spaces.
Namely, we show in Corollary \ref{radio de convergencia}  that the radius of convergence of a holomorphic function in several complex variables can be computed in terms of the norms of the symmetric multilinear mappings associated to the polynomials of the Taylor series expansion.

On the other hand, we prove that $c(\ell_1^d(\mathbb R)) > 1$ showing that Theorem \ref{teo fin dim} is no longer valid in the real case. 
In addition, we show that for finite dimensional real spaces the polarization constant coincides with, what we call, the Bochnak's complexification constant, and therefore is bounded by 2. 

All the results that appear above are treated in Section \ref{Finite dimensional spaces}. We also deal with some other problems related with polarization constants.

The aforementioned result of S. Banach (for the particular case $k=2$) says that $\ccc(2,\mathcal H)=1$, if $\mathcal{H}$ is a Hilbert space. Note that this is equivalent to the well-known identity valid for every self-adjoint operator $T \in \mathcal L (\mathcal{H})$:
$ \Vert T \Vert= \sup_{\Vert \mathbf{x} \Vert = 1} \vert \langle T\mathbf{x}, \mathbf{x} \rangle \vert$. 
This equality can also be reinterpreted for general Banach spaces $X$, which is again equivalent to the fact that $\ccc(2,X)=1$.
In the real case, the equality $\ccc(2,X)=1$ forces $X$ to be a Hilbert space, see \cite{benitez1993characterization}.
In the complex setting, there are non Hilbert spaces satisfying the above property. We show, in Section \ref{Type and Cotype and the symmetric operator norm property} that, in terms of type and cotype  those spaces $X$ with $\ccc(2,X)=1$ ``look like'' Hilbert spaces.

Additionally, we relate polarization constants with certain estimates of the nuclear norm of the product of functionals/polynomials. 
Recall that a $k$-homogeneous polynomial $P\in \mathcal P(^kX)$ is nuclear if  there exist bounded sequences $(\varphi_j)_j\in X^*$ and $(\lambda_j)_j\in\ell_1$ such that
\begin{equation}\label{ec nuclear}
 P(\mathbf{x})=\sum_{j=1}^\infty \lambda_j \varphi_j(\mathbf{x})^k,\quad\textrm{for all }\mathbf{x}\in X.   
\end{equation}
 The space $\mathcal P_N(^kX)$ of nuclear $k$-homogeneous polynomials on $X$ is a Banach space when endowed with the norm
$$\|P\|_{\mathcal P_{N}(^{k}X)}=\inf\{\sum_{j=1}^\infty |\lambda_j| \|\varphi_j\|^k\},
$$
where the infimum is taken over all the representations of $P$ as in \eqref{ec nuclear}.

We show in Section \ref{nuclear} that if $X^*$ has the approximation property, then $\ccc(k,X^*)$ is exactly  the best constant $C>0$ such that for any functionals $\varphi_1, \dots, \varphi_k \in X^*$ the following inequality holds
\begin{equation}\label{producto}
\Vert \varphi_1 \cdots \varphi_k\Vert_{\mathcal P_{N}(^{k}X)}\leq  C  \Vert \varphi_1 \Vert \cdots \Vert \varphi_k \Vert,
\end{equation}
where $\varphi_1 \cdots \varphi_k$ is the $k$-homogeneous polynomial given by the pointwise product of the linear functionals. 

Moreover, we study the best constant $\mmm(k_1,\ldots,k_n,X)$ such that for any nuclear homogeneous polynomials $P_1, \ldots, P_n$  of degrees $k_1,\ldots, k_n$ respectively, we have that
\begin{equation}\label{problema}
\Vert P_1 \cdots P_n\Vert_{\mathcal P_{N}(^{k}X)}\leq  \mmm(k_1,\ldots,k_n,X)  \Vert P_1 \Vert_{\mathcal P_{N}(^{k_1}X)} \cdots \Vert P_n \Vert_{\mathcal P_{N}(^{k_n}X)},
\end{equation}
where, as before, $P_1 \cdots P_n$ is the homogeneous polynomial of degree $k = \sum_{i=1}^nk_i$ given by pointwise product; and show that $\mmm(k_1,\ldots,k_n,X)$ is intimately linked with the  polarization constants. 
Note that the best constant $C>0$ that fulfills Equation \eqref{producto} is exactly $\mmm(k_1,\ldots,k_n,X)$ for $n=k$ and $k_i=1$ for all $1 \leq i \leq k.$

It is important to remark that  formula  \eqref{producto} considers the \emph{nuclear norm} of the product of linear functionals. The reader should not mistake this with estimating the \emph{uniform norm} of the product of linear functionals. This analogous problem involves the \emph{linear polarization constant},  whose name is similar to the constant studied in this article and may cause some confusion. For more information on the linear polarization constant we refer the reader to the articles \cite{anagnostopoulos2006polarization, benitez1998lower, carando2017linear, pappas2004linear, revesz2004plank} and the references therein.

\section{Finite dimensional spaces}\label{Finite dimensional spaces}

The key ingredient to prove Theorem \ref{teo fin dim} is to treat first the case where $X=\ell_1^d(\mathbb C)$, the complex $\ell_1$-space of dimension $d$ (which is expected to be the worst one).
Our argument will heavily rely on the following result of Sarantopoulos \cite[Proposition 4]{S2}:
\begin{equation}\label{Prop Sarant finite dim}
\ccc(k, \ell_1^d(\mathbb C))=\max \left\{\frac{k_1!\cdots k_d!}{k !}\frac{k^k}{k_1^{k_1}\cdots k_d^{k_d}}: k_1+\cdots + k_d =k\right\}.
\end{equation}

\begin{prop}\label{ele-1} For any non negative integer $d$,  $\ccc(\ell_1^d(\mathbb C))=1$.
\end{prop}

\begin{proof} For any non negative integer $m$, the maximum of the set
$\left\{ \frac{i!j!}{i^ij^j} : i,j\in \zN,\ i+j=m\right\}$
is attained at $i=\frac{m}{2}$, $j=\frac{m}{2}$ if $m$ is even, and at $i=[\frac{m}{2}]+1$, $j=[\frac{m}{2}]$ if $m$ is odd. This can be deduce, for example, from the fact that if $i>j$ then
$$\frac{i!j!}{i^ij^j} \leq \frac{(i-1)!(j+1)!}{(i-1)^{i-1}(j+1)^{j+1}}.$$ Indeed, this is equivalent to
$$\left(\frac{j+1}{j}\right)^j \leq  \left(\frac{i}{i-1}\right)^{i-1},$$
which holds because  $\left( \frac{x+1}{x}\right)^x$ is an increasing function.

From this we derive that the maximum in \eqref{Prop Sarant finite dim} is attained at $\underbrace{c+1,\ldots,c+1}_{r\,\,\, times }, \underbrace{c,\ldots,c}_{d-r\,\,\, times }$, where $c,r\in \zN_0$ are such that $k=d\,c+r$ and $0\leq r <d$. In other words,

\begin{equation} \label{maximum}
\ccc(k,\ell_1^d(\mathbb C)) = \frac{(c+1)!^r c!^{d-r}}{(c+1)^{(c+1)r} c^{c(d-r)}}\frac{k^k}{k !}.    
\end{equation}

Now, since $\left(\frac{k^k}{k!}\right)^{\frac1k}\xrightarrow[k\rightarrow \infty]{}e$, in order to prove that $\ccc(k,\ell_1^d(\mathbb C))^{\frac1k}\xrightarrow[k\rightarrow \infty]{}1$ we need to check, for $r=0,\ldots, d-1$, that $$\left(\frac{(c+1)!^r c!^{d-r}}{(c+1)^{(c+1)r} c^{c(d-r)}}\right)^\frac{1}{d\,c+r} \xrightarrow[c\rightarrow \infty]{} \frac1e.$$
Indeed,
\begin{eqnarray*}
\frac{(c+1)!^r c!^{d-r}}{(c+1)^{(c+1)r} c^{c(d-r)}}&=& \frac{c!^r c!^{d-r}}{(c+1)^{cr} c^{c(d-r)}} =\frac{c!^d}{(c+1)^{cr}c^{c(d-r)}}\\
&= & \frac{c!^d}{c^{cd}} \left(\frac{c}{c+1}\right)^{cr}= \left(\frac{c!}{c^c}\right)^{d}\left(\frac{c}{c+1}\right)^{cr}.
\end{eqnarray*}
Therefore,
$$
\left(\frac{(c+1)!^r c!^{d-r}}{(c+1)^{(c+1)r} c^{c(d-r)}}\right)^\frac{1}{d\,c+r}= \left[\left(\frac{c!}{c^c}\right)^{\frac1c}\right]^{\frac{cd}{dc+r}} \left(\frac{c}{c+1}\right)^{\frac{cr}{dc+r}}\xrightarrow[c\rightarrow \infty]{} \frac1e,
$$
which completes the proof.
\end{proof}

Bellow we give an alternative proof, which is shorter, due to one of the anonymous referees of this article. The up side of the original proof is that gives the exact value of the maximum on \eqref{Prop Sarant finite dim}, which is explicitly written in \eqref{maximum}. 
\begin{proof}[Alternative proof of Proposition \ref{ele-1}] We use the following inequality due to Stirling formula
\begin{equation*}
    \sqrt{2\pi}\, j^{j+\frac{1}{2}}e^{-j} \leq j! \leq e\,j^{j+\frac{1}{2}}e^{-j}.
\end{equation*}

This, combined with the arithmetic-geometric mean inequality, gives
\begin{eqnarray}
\frac{k_1!\cdots k_d!}{k !}\frac{k^k}{k_1^{k_1}\cdots k_d^{k_d}} &\leq& \frac{e^d}{\sqrt{2\pi}k^{\frac{1}{2}}}(k_1\cdots k_d)^{\frac{1}{2}} \nonumber\\
&\leq & \frac{e^d}{\sqrt{2\pi} k^{\frac{1}{2}}} \left( \frac{k_1+\cdots + k_d}{d} \right)^{\frac{d}{2}} \nonumber\\ 
&\leq & \frac{e^d}{\sqrt{2\pi}d^\frac{d}{2}}k^\frac{d-1}{2}.\nonumber\
\end{eqnarray}
Therefore we conclude $\ccc(\ell_1^d(\mathbb C))=\limsup_{k\to \infty} \ccc(k, \ell_1^d(\mathbb C))^\frac{1}{k} = 1$.
\end{proof}

The following lemma, which is surely known, asserts that every finite dimensional space is  ``almost'' a quotient of a finite dimensional $\ell_1$-space. We include a simple proof since we could not find a proper reference.

\begin{lem}\label{lem cociente l1} Given a finite dimensional Banach space $X$ and $\varepsilon>0$, there is $d=d(\varepsilon,X)\in \zN$ and a norm one surjective linear operator $\mathbf{q}:\ell_1^d\rightarrow X$ such that for every $\mathbf{x}\in X$ there is $\mathbf{z}\in  \ell_1^d$ with $\mathbf{q}(\mathbf{z})=\mathbf{x}$ and $\Vert \mathbf{z}\Vert_1 < (1+\varepsilon)\Vert x \Vert$.
\end{lem}
\begin{proof}
Take $0<\eta<1$ such that $\frac{1}{1-\eta}< (1+\varepsilon)$. Let $\{\mathbf{h}_1,\ldots, \mathbf{h}_d\}\subseteq S_X$ be an $\eta-$net. Let us define $\mathbf{q}:\ell_1^d\rightarrow X$ over the elements of the canonical basis $\{\mathbf{e}_1,\ldots, \mathbf{e}_d\}$ of $\ell_1^d$ as $\mathbf{q}(\mathbf{e}_j)=\mathbf{h}_j.$ By the triangle inequality, $\Vert \mathbf{q} \Vert \leq 1$.

Now, fixed $\mathbf{x}\in S_X$ we need to find $\mathbf{z}\in  \ell_1^d$ such that $\mathbf{q}(\mathbf{z})=\mathbf{x}$ and $\Vert \mathbf{z}\Vert_1 < (1+\varepsilon)$. Take $\delta_1=1$. Let $\mathbf{h}_{n_1}$ be an element of the $\eta-$net such that $$\delta_2:=\Vert \mathbf{x}-\mathbf{h}_{n_1} \Vert < \eta.$$ Now take $\mathbf{h}_{n_2}$ such that $$\delta_3:=\Vert (\mathbf{x}-\mathbf{h}_{n_1})-\delta_2 \mathbf{h}_{n_2} \Vert < \delta_2 \eta < \eta^2.$$ Following this process we construct a sequence $(\mathbf{h}_{n_j})_{j\in \zN}$ such that
$$\left\Vert \mathbf{x} - \sum_{j=1}^{m+1} \mathbf{h}_{n_j} \right\Vert < \eta^m. $$
Clearly $\mathbf{x}=\sum_{j=1}^\infty \delta_j  \mathbf{h}_{n_j}$ and therefore, if we take $\mathbf{z}=\sum_{j=1}^\infty \delta_j  \mathbf{e}_{n_j}$, we have that $\mathbf{x}=\mathbf{q}(\mathbf{z})$ and
$$
\Vert \mathbf{z}\Vert \leq \sum_{j=1}^\infty \delta_j 
<  \sum_{j=1}^\infty \eta^{j-1} 
<  \frac{1}{1-\eta}< 1+\varepsilon,  
$$
which concludes the proof.
\end{proof}

Now, we are ready to prove Theorem \ref{teo fin dim}.

\begin{proof}[Proof of Theorem \ref{teo fin dim}.]

Let $X$ be a finite dimensional complex Banach space. Given $\varepsilon>0$, we first show there is $d=d(\varepsilon,X)\in \zN$ such that
\begin{equation}\label{acotacion c}
    \ccc(k, X)\leq (1+\varepsilon)^k \ccc(k, \ell_1^d(\mathbb C)).
\end{equation}

Indeed, given a $k$-homogeneous polynomial $P$ and $\mathbf{x}_1,\ldots, \mathbf{x}_k\in S_X$, we need to see that
$$|\overset{\vee}{P}(\mathbf{x}_1,\ldots, \mathbf{x}_k)| \leq (1+\varepsilon)^k \ccc(k,\ell_1^d(\mathbb C)) \Vert P \Vert.$$

Let $\mathbf{q}:\ell_1^d(\mathbb C)\rightarrow X$ be as in the previous lemma. Take $\mathbf{z}_1,\ldots,\mathbf{z}_k\in \ell_1^d(\mathbb C)$ such that $\mathbf{q}(\mathbf{z}_j)=\mathbf{x}_j$ and $\Vert \mathbf{z}_j\Vert < 1+\varepsilon$. Note that the multilinear form $\overset{\vee}{P}\circ (\mathbf{q},\ldots, \mathbf{q})$ has norm less than or equal to one and also its associated polynomial is just $P\circ \mathbf{q}$. Then we have
\begin{eqnarray}
|\overset{\vee}{P}(\mathbf{x}_1,\ldots, \mathbf{x}_k)| &= & |\overset{\vee}{P}\circ (\mathbf{q},\ldots,\mathbf{q}) (\mathbf{z}_1,\ldots, \mathbf{z}_k)|\nonumber  \\
&\leq & \Vert P \circ \mathbf{q}\Vert \Vert \mathbf{z}_1\Vert \cdots \Vert \mathbf{z}_k\Vert \ccc(k, \ell_1^d(\mathbb C))\\
&< & \Vert P \Vert(1+\varepsilon)^k c(k, \ell_1^d(\mathbb C)). \nonumber \
\end{eqnarray}

Thus, using Proposition \ref{ele-1} the proof of the theorem follows from inequality \eqref{acotacion c}.

\end{proof}

\subsection{Consequences}

Let $X$ and $Y$ be normed spaces and $U \subset X$ be an open set. Recall that a function $f : U \to Y$  is holomorphic if it is Fr\'echet differentiable at every point of $U$. 
We denote by $H(U;Y)$ the space of holomorphic functions from $U$ to $Y$. Given $f \in H(U;Y)$ for each $\mathbf{a} \in U$  there is a sequence of  polynomials $P_k$, $k=0, 1, 2, \dots,$ with $P \in \mathcal P(^k X;Y)$ such that
\begin{align}
    f(\mathbf{x}) = \sum_{k=0}^\infty P_k(\mathbf{x}-\mathbf{a})
\end{align}
uniformly in a ball centered at $\mathbf{a}$ contained in $U$. 

The supremum of all $r>0$ such that the series converges uniformly on the ball $B(\mathbf{a},r)$ is called the radius of convergence and can be computed by the Cauchy-Hadamard formula
\begin{align}
    R^\mathbf{a}(f)= \frac{1}{\limsup_{k\to \infty} \Vert P_k \Vert^{1/k}}.
\end{align}

For our purposes it is also interesting to consider the following value:
\begin{align}
R_{\mbox{mult}}^\mathbf{a}(f):=\frac{1}{\limsup_{k\to \infty} \Vert \overset{\vee}{P}_k \Vert^{1/k}}.
\end{align}

It is clear that $R_{\mbox{mult}}^\mathbf{a}(f) \leq R^\mathbf{a}(f)$ for every $f$ and $\mathbf{a}$.
The following result characterizes a reverse inequality in terms of the polarization constant of $X$.
\begin{prop}
Let $X$ be a normed space and $\mathbf{a} \in X$. 
Then, the polarization constant $\ccc(X)$ is the minimum of all $C>0$ such that
\begin{equation}\label{relacion radios}
R^{\mathbf{a}}(f) \leq C R_{mult}^{\mathbf{a}}(f),
\end{equation}
for every normed space  $Y$ and every $Y$-valued holomorphic function $f$ defined in a neighborhood of $\mathbf{a}$.

\end{prop}

\begin{proof}
It is enough to prove the case where $\mathbf{a}=0$; for simplicity we denote $R^0(f)=R(f)$ and $R^0_{mult}(f)=R_{mult}(f)$. Let $I(X)$ be the minimum of all $C>0$ such that Equation \eqref{relacion radios} holds. 
It easy to see that
\begin{equation}
    \Vert \overset{\vee}{P} \Vert \leq \ccc(k,X) \Vert P \Vert, \; \mbox{for all } P \in \mathcal P(^kX;Y) \mbox{ and \emph{any} normed space $Y$.}
\end{equation}
Let $f \in H(U;Y)$ with Taylor expansion
$f(\mathbf{x}) = \sum_{k=0}^\infty P_k(\mathbf{x})$. Given $\varepsilon>0$ we have
\begin{align}\label{acotacion normas}
 \Vert P_k \Vert \leq \Vert \overset{\vee}{P}_k \Vert \leq (\ccc(X)+\varepsilon)^k \Vert P_k \Vert,
\end{align}
for  $k$ large enough.
Then,
$$ \frac{R(f)}{\ccc(X)+\varepsilon} \leq R_{\mbox{mult}}(f) \leq R(f),$$
for every $\varepsilon>0$, therefore $R(f) \leq \ccc(X) R_{\mbox{mult}}(f)$ and $I(X) \leq \ccc(X)$. 

Suppose $I(X)<\delta<\ccc(X)$, then there is a sequence of degrees $(k_j)_{j \in \mathbb N}$  such that $\ccc(k_j,X)^{1/k_j}>\delta$ for all $j$. 
Now, for each $j \in \mathbb N$ pick a norm one polynomial
$P_{k_j} \in \mathcal P(^{k_j}X)$ such that $\Vert \overset{\vee}{P}_{k_j} \Vert \geq \delta^{k_j}.$ 
Hence for $f = \sum_{j=1}^{\infty} P_{k_j} \in H(X)$, we have $R(f)=1$ and $R_{\mbox{mult}}(f)< \frac{1}{\delta}$. This provides a contradiction since 
\begin{equation}
    1=R(f) \leq I(X) R_{\mbox{mult}}(f)< \frac{I(X)}{\delta} < 1.
\end{equation}
Therefore $I(X)=\ccc(X).$
\end{proof}

Observe that the previous proposition implies that $\ccc(X)=1$ if and only if 
\begin{equation}
R^{\mathbf{a}}(f) = R_{\mbox{mult}}^{\mathbf{a}}(f),
\end{equation}
for every $\mathbf{a} \in X$, every normed space  $Y$ and every $f \in  H(U;Y)$, where $U$ is an open set containing $\mathbf{a}$.
Thus, as a consequence of Theorem \ref{teo fin dim} we obtain the following corollary.

\begin{cor} \label{radio de convergencia}
Let $X$ be a finite dimensional space and $Y$ be an arbitrary normed space. For each $\mathbf{a} \in X$ we have $R^\mathbf{a}(f)=R_{\mbox{mult}}^\mathbf{a}(f)$ for every holomorphic function $f \in H(U;Y)$, where $U$ is any open set containing $\mathbf{a}$.
\end{cor}

Let $X$ be an $n$-dimensional space and $Y$ be a normed space. Each polynomial $P \in \mathcal P(^kX;Y)$ can be written as
$$ P(\mathbf{z}) = \sum_{\vert \alpha \vert =k} c_{\alpha} z_1^{\alpha_1} \dots z_n^{\alpha_n},$$
where $\mathbf{z}=(z_1, \dots, z_n).$ Therefore if the monomial expansion 
$$
\sum_{k=0}^\infty \sum_{\vert \alpha \vert =k} c_{\alpha} z_1^{\alpha_1} \dots z_n^{\alpha_n}
$$
converges uniformly and absolutely on a given set then the same happens to the power series $\sum_{k=0}^\infty P_k(\mathbf{z})$, where $P_k(\mathbf{z})=\sum_{\vert \alpha \vert =k} c_{\alpha} z_1^{\alpha_1} \dots z_n^{\alpha_n}.$

A reciprocal result was proved in \cite[Proposition 4.6]{mujica2010complex}. Following his arguments together with Equation \eqref{acotacion normas} and applying our main result we are allowed to expand the region of convergence from $\frac{R}{e} B_{\ell_1^n(\mathbb C)}$ to $RB_{\ell_1^n(\mathbb C)}$. 

\begin{prop}
Let $X$ be an $n$-dimensional space and $Y$ be a normed space. Consider  a power series from $X$ into $Y$, $\sum_{k=0}^\infty P_k(\mathbf{x})=\sum_{k=0}^\infty \overset{\vee}{P}_k(\mathbf{x}^k)$, with radius of convergence $R>0$. Given unitary vectors $\mathbf{e}_1, \dots, \mathbf{e}_n \in X$, set
$$
c_{\alpha}=\frac{k!}{\alpha!} \overset{\vee}{P}_k(\mathbf{e}_1^{\alpha_1},  \dots, \mathbf{e}_n^{\alpha_n}),
$$
for each $\alpha \in \mathbb N_0^{(\mathbb N)}$ with $\vert \alpha \vert =k$. Then we have, for $\Vert (z_1, \dots, z_n) \Vert_{\ell_1^n(\mathbb C)} \leq R$, the equality
$$
\sum_{k=0}^\infty P_k(z_1 \mathbf{e}_1 + \dots + z_n \mathbf{e}_n) = \sum_{\alpha} c_{\alpha} z_1^{\alpha_1} \dots z_n^{\alpha^n},
$$
and both series converge absolutely and uniformly for $\Vert (z_1, \dots, z_n) \Vert_{\ell_1^n(\mathbb C)} \leq r$ where $0 < r < R.$
\end{prop}

\subsection{Real case}  \label{real case}

Now we consider finite dimensional {\it real} spaces. One could speculate that the polarization constant behaves as in the finite dimensional complex case or as in the infinite dimensional real case. Nevertheless, none of these situations actually happen.

We exhibit below an example showing that a finite dimensional real normed space can have polarization constant bigger than 1, contrary to what we have proved in Theorem \ref{teo fin dim} for complex spaces. We also  show that each finite dimensional real normed space has polarization constant less than or equal to 2 (recall that in the infinite dimensional case the upper bound $e$ cannot be improved). 

In order to do this we  rely heavily on the complexification procedure for Banach spaces. We refer the reader to \cite{munoz1999complexifications} and the references therein for information on this subject.

We just recall that if $X$ is a real Banach space and $\widetilde X$ denotes its Bochnak's complexification then for every multilinear form $L \in \mathcal L (^k X)$, it holds that $\Vert L \Vert = \Vert \widetilde L \Vert,  $ where $\widetilde L \in \mathcal L (^k \widetilde X)$ stands for the complexified form. Given a polynomial $P \in \mathcal P (^k  X)$, we will also refer to $\widetilde P \in \mathcal P (^k \widetilde X)$ to its complexified mapping.

For each $k$ we denote by $\bbb(k,X)$ the best constant $C>0$ satisfying
$$
\|\widetilde P\|\le C \|P\|, \quad \forall P\in\mathcal P(^kX)
$$ 
and we call the \emph{Bochnak constant} of $X$ as $$\bbb(X) = \limsup_{k \to \infty} \bbb(k,X)^{1/k}.$$

We begin by bounding the polarization constant of a 2-dimensional real $\ell_1$-space:

\begin{prop}
$\sqrt[4]{2}\le \ccc(\ell_1^2(\mathbb R))\le\sqrt 2$.
\end{prop}

\begin{proof}If $d(X,Y)$ is the Banach-Mazur distance between two isomorphic Banach spaces $X$ and $Y$, then
$$\ccc(k,X)\leq \ccc(k,Y) (d(X,Y))^k$$
The proof of this result is similar to the proof of \cite[Lemma 12]{benitez1998lower}. Since the Banach-Mazur distance between $\ell_1^2(\mathbb R)$ and $\ell_2^2(\mathbb R)$ is $\sqrt{2}$ and $\ccc(\ell_2^2(\mathbb R))=1$ we get $\ccc(\ell_1^2(\mathbb R))\le\sqrt 2$. 

For the lower bound we see that $\ccc(8m,\ell_1^2(\mathbb R)) \geq 2^{2m-1}$ for every $m \geq 1.$
Let $P \in \mathcal P(^{8m}\ell_1^2(\mathbb R))$ given by 
\begin{align}
    P(x,y)  & = (xy)^{2m} \sum_{j=0}^{2m} \binom{4m}{2j} (-1)^j y^{2j} x^{4m-2j} \label{expresion P}  \\
    &= (xy)^{2m} \frac{(x+iy)^{4m}+(x-iy)^{4m}}{2}  \nonumber\\
       & = (xy)^{2m} \operatorname{Re}(x+iy)^{4m}. \label{expresion P2}
\end{align}
Then it is standard to see that for a unit vector $(x,y)$ in $\ell_1^2(\mathbb R)$,
\begin{align*}
    \vert P(x,y) \vert & \leq  \vert xy \vert^{2m} (x^2+y^2)^{2m}
    \leq \frac{1}{{2}^{6m}}.
\end{align*}
Also, since $\vert P(\frac{1}{2},\frac{1}{2}) \vert =\frac{1}{{2}^{6m}}$ we have that $\Vert P \Vert =\frac{1}{{2}^{6m}}. $

Now using that Bochnak's complexification of $\ell_1^2(\mathbb R)$ is $\ell_1^2(\mathbb C)$, we get that the complexified polynomial $\widetilde P \in \mathcal P(^{8m}\ell_1^2(\mathbb C))$ has the same expression as in \eqref{expresion P} (notice that the alternative expression given in \eqref{expresion P2} is not valid for  the complexified polynomial).
Also, since $ \left| \widetilde P \left(\frac{1}{2},\frac{i}{2} \right) \right|=\frac{1}{2^{4m+1}},$
then
$$
\| \widetilde P \| \geq 2^{2m-1}\left\| P \right\|.
$$

Therefore 
$$
\Vert\overset{\vee}{P} \Vert = \Vert \widetilde{\overset{\vee}{P}} \Vert = \Vert \overset{\vee}{\widetilde P} \Vert \geq \Vert \widetilde P \Vert \geq 2^{2m-1} \Vert P \Vert,
$$ 
and so $\ccc(8m,\ell_1^2(\mathbb R)) \geq 2^{2m-1}$.
\end{proof}

The same example as in the proposition works for any finite dimensional $\ell_1$-space (just considering the first two coordinates). Hence, for every dimension $d$,
$$
\ccc(\ell_1^d(\mathbb R))\ge \sqrt[4]{2}.
$$

Now we prove that, for a real finite dimensional space $X$, the polarization and Bochnak constants concide and they are smaller than 2.

\begin{prop}
For any finite dimensional real normed space $X$ it holds
$$
\ccc(X)= \bbb(X)\le 2 .
$$

\end{prop}

\begin{proof}
By \cite[Proposition 18]{munoz1999complexifications}, we know that $\bbb(k,X) \leq 2^{k-1}$ and therefore $\bbb(X) \leq 2.$
Let us show that $\ccc(X)= \bbb(X)$. For every polynomial $P \in \mathcal P(^k X),$
$$
\Vert\overset{\vee}{P} \Vert = \Vert \widetilde{\overset{\vee}{P}} \Vert = \Vert \overset{\vee}{\widetilde P} \Vert  \leq \ccc(k,\widetilde X) \Vert \widetilde P \Vert \leq \ccc(k,\widetilde X) \bbb(k, X)\Vert P \Vert.
$$ 
This implies that
$$
\ccc(k,X) \leq \ccc(k,\widetilde X) \bbb(k, X),
$$
thus, 
$$
\ccc(X) \leq \ccc(\widetilde X) \bbb(X).
$$
Since $\tilde X$ is a finite dimensional complex space, by Theorem \ref{teo fin dim} we know that $\ccc(\widetilde X)=1$ and therefore $\ccc(X) \leq \bbb(X).$

On the other hand, for each  $P \in \mathcal P(^kX)$, 
$$
\Vert \widetilde P \Vert \le  \Vert \overset{\vee}{\widetilde P} \Vert = \Vert \widetilde{\overset{\vee}{P}} \Vert =\Vert\overset{\vee}{P} \Vert \leq  \ccc(k,X) \Vert P \Vert.
$$ 
Then $$\bbb(k,X)\le\ccc(k,X),\ \forall k\in\mathbb N$$
and the result follows.
\end{proof}





\section{Type and Cotype and the symmetric operator norm property}\label{Type and Cotype and the symmetric operator norm property}

It is standard that for a Hilbert space $\mathcal{H}$ and a self-adjoint operator $T \in \mathcal L (\mathcal{H})$ we have the equality
\begin{equation}\label{propiedad hilbert}
    \Vert T \Vert= \sup_{\Vert \mathbf{x} \Vert = 1} \vert \langle T\mathbf{x}, \mathbf{x} \rangle \vert.
\end{equation}
The notion of self-adjoint operator in a Hilbert space can be extended to operators from an arbitrary Banach space $X$ to its dual $X^*$. Namely, $T \in \mathcal L(X,X^*)$ is \emph{symmetric} if $T(\mathbf{x})(\mathbf{y})=T(\mathbf{y})(\mathbf{x})$, for all $\mathbf{x},\mathbf{y} \in X$.
We say that a Banach space $X$ has the \emph{symmetric operator norm property} if for every symmetric $T \in \mathcal L(X,X^*)$,
\begin{equation}\label{symmetric operator norm property}
    \Vert T \Vert= \sup_{\Vert \mathbf{x} \Vert = 1} \vert T(\mathbf{x})(\mathbf{x}) \vert.
\end{equation}

Note that the symmetric operator norm property for $X$ can be restated as $\ccc(2,X)=1.$
Indeed, each bilinear symmetric  form $B$ in $X$ can be isometrically identified with a symmetric linear operator $T_B \in \mathcal L(X,X^*)$ by $T_B(\mathbf{x})(\mathbf{y})=B(\mathbf{x},\mathbf{y})$, for every $\mathbf{x},\mathbf{y} \in X.$ Then Equation \eqref{symmetric operator norm property} says that the norm of $B$ coincides with the norm of its associated 2-homogeneous polynomial.

For real spaces, having the symmetric operator norm property is equivalent to being a Hilbert space \cite[Proposition 2.8]{benitez1993characterization}. In the complex case, this is no longer true: in \cite[Proposition 3]{S2} it is shown that if $\mathcal{H}$ is a Hilbert space then $\mathcal{H} \oplus_{\infty} \mathbb C$ also enjoys the symmetric operator norm property. Here we provide necessary conditions related to the notion of type and cotype,  for a complex Banach space to have this property. For an introduction on the concepts of type and cotype we refer to Maurey's survey \cite{maurey2003type}.
 Given a Banach space $X$, we denote by $\ppp(X)$ and $\qqq(X)$ the constants
$$\ppp(X)=\sup\{r: X \mbox{ has type } r\}$$
$$\qqq(X)=\inf\{r: X \mbox{ has cotype } r\}.$$

Recall that a classical result of Kwapie\'n \cite{kwapien1972isomorphic} states that a Banach space has type $2$ and cotype $2$ if and only if it is a Hilbert space. 
The following result shows that if $\ccc(2,X)=1$ then $X$ is, in some sense, similar to a  Hilbert space. 

\begin{thm}\label{teo tipo y cotipo} Let $X$ be an infinite dimensional complex Banach space with the symmetric operator norm property. Then, $$\ppp(X)=\qqq(X)=2.$$
\end{thm}

It should be noted that the conclusion of Theorem \ref{teo tipo y cotipo} (i.e., $\ppp(X)=\qqq(X)=2$) holds trivially for every finite dimensional normed space. This is why the statement is just given for the infinite dimensional case.

To prove the theorem, we need to show that for each $p \neq 2$, there exists a dimension $d=d(p)$ such that $\ccc(2, \ell_p^d(\mathbb C)) >1.$ For $p < 2$ this was proved by Sarantopoulos:

\begin{lem}{ \cite[Theorem 2]{S1}.} For $1\leq p <2$ we have that
$$\ccc(2,\ell_p^2(\mathbb C)) >1.$$
\end{lem}

To obtain $\ccc(2, \ell_p^3(\mathbb C)) >1$ for $p >2$ we need an interpolation result.
Polynomials can easily be interpolated by means of multilinear forms' interpolation \cite[Theorem 4.4.1]{bergh2012interpolation} (at the cost of the polarization constant). Since we need to avoid polarization constants we prove the following proposition which could be interesting in its own right.
We refer to \cite{bergh2012interpolation} for an introduction and the notation we use on complex interpolation of Banach spaces.

\begin{prop}\label{prop interpolacion} Let $\overline{A}=(A_0, A_1)$ and  $\overline{B}=(B_0, B_1)$ be compatible Banach couples. Let $P:A_i\rightarrow B_i$ be a  $k$-homogeneous polynomial with norm  $\Vert P\Vert_i=M_i$ for $i=0,1$. For $0<\theta <1$, $P$ extends to a unique continuous  $k$-homogeneous polynomial
$$P:[A_0,A_1]_\theta \rightarrow [B_0,B_1]_\theta$$ with norm at most
$$\Vert P \Vert_{\PP([A_0,A_1]_\theta, [B_0,B_1]_\theta)}\leq M_0^{1-\theta}M_1^\theta.$$
\end{prop}

\begin{proof} The unique extension of $P$ follows by applying  \cite[Theorem 4.4.1]{bergh2012interpolation} to $\overset{\vee}{P}$. We only need to check the upper bound for the norm.

Take $\mathbf{a}\in [A_0,A_1]_\theta$ with $\Vert \mathbf{a}\Vert_{[A_0,A_1]_\theta} <1$, then there is $f\in \mathfrak {F}(\overline{A})$ such that $f(\theta)=\mathbf{a}$ and $$ \max\{ \sup_{t\in \zR}\Vert f(it)\Vert_{A_0}, \sup_{t\in \zR}\Vert f(1+it)\Vert_{A_1}\}=\Vert f\Vert_\mathfrak{F} <1 .$$

Let
$$g:\{z\in\zC: 0\leq \operatorname{Re}(z)\leq 1\} \rightarrow B_0 +B_1$$
defined as  $g(z)=M_0^{z-1}M_1^{-z}P(f(z))$. Since $P$ is a polynomial, we have that $g\in \mathfrak {F}(\overline{B})$. Moreover
$$\Vert g(it) \Vert_{B_0}\leq M_0^{-1} \Vert P(f(it)) \Vert_{B_0} \leq \Vert f(it) \Vert_{A_0}^k\leq \Vert f \Vert_\mathfrak{F}^k <1$$
$$\Vert g(1+it) \Vert_{B_1}\leq M_1^{-1} \Vert P(f(1+it)) \Vert_{B_1} \leq \Vert f(1+it) \Vert_{A_1}^k\leq \Vert f \Vert_\mathfrak{F}^k <1.$$
Therefore $\Vert g \Vert \leq 1$. Thus
$$\Vert M_0^{\theta-1}M_1^{-\theta}P(\mathbf{a})\Vert_{[B_0,B_1]_\theta} = \Vert M_0^{\theta-1}M_1^{-\theta}P(f(\theta))\Vert_{[B_0,B_1]_\theta} = \Vert g(\theta)\Vert_{[B_0,B_1]_\theta} \leq \Vert g \Vert_\mathfrak{F}\leq 1.$$
which completes the proof.
\end{proof}

We use this interpolation result to prove the next auxiliary lemma.

\begin{lem}\label{lema 2} For $2< p \leq \infty$ we have that
$$\ccc(2,\ell_p^3(\mathbb C)) >1.$$
\end{lem}

\begin{proof}Let $P: \mathbb C^3 \to \mathbb C$ be the $2$-homogeneous polynomial defined by $$P(z_1,z_2,z_3)=z_1^2+z_2^2+z_3^2-2z_1z_2-2z_1z_3-2z_2z_3.$$
In \cite[Addendum]{Var74} it is shown that $\Vert P \Vert_{\mathcal P(^2 \ell_\infty^3)} =5$.

Note that the symmetric bilinear mapping associated to $P$ is given by
\begin{equation}
 \overset{\vee}{P}((w_1,w_2,w_3),(z_1,z_2,z_3))   = \begin{bmatrix}
w_1 & w_2 & w_3\\
\end{bmatrix}
\begin{bmatrix} 
\;\;1 & -1 & -1  \\
-1 & \;\;1 & -1 \\
-1 & -1 & \;\;1
\end{bmatrix}
\begin{bmatrix}
z_1 \\
z_2 \\

z_3
\end{bmatrix}.
\end{equation}
If we set $\lambda:=e^{\frac{2\pi i}{3}}$, there exists  $(w_1,w_2,w_3)$ with $\vert w_1 \vert = \vert w_2 \vert =\vert w_3 \vert =1$ such that
\begin{eqnarray}
\overset{\vee}{P}((w_1,w_2,w_3),(1,\lambda,\lambda^2))=
|1-\lambda-\lambda^2|+|-1+\lambda-\lambda^2|+|-1-\lambda+\lambda^2|=6.\nonumber \
\end{eqnarray}

Therefore $$\Vert \overset{\vee}{P} \Vert_{\mathcal L(^2 \ell_p^3)}  \geq \frac{6}{\Vert(w_1,w_2,w_3)\Vert_p \Vert(1,\lambda,\lambda^2)\Vert_p} =\frac{6}{3^\frac{2}{p}}.$$
Thus, we only need to see that $$\Vert P \Vert_{\mathcal P(^2 \ell_p^3)} <  \frac{6}{3^\frac{2}{p}}.$$

Note that 
$$\Vert P \Vert_{\mathcal P(^2 \ell_2^3)}=\Vert \overset{\vee}{P} \Vert_{\mathcal L(^2 \ell_2^3)}  =2,$$
since the eigenvalues of the associated symmetric matrix are $-1,2,2$.

Let $0<\theta <1$ such that
$$\frac{1}{p}= \frac{1-\theta}{2}+\frac{\theta}{\infty}.$$
Observe that  $\ell_p^3(\mathbb C)= [\ell_2^3(\mathbb C), \ell_\infty^3(\mathbb C)]_{\theta}$ so,  by Proposition \ref{prop interpolacion}, we get
$$
\Vert P \Vert_{\mathcal P(^2 \ell_p^3)} \leq \Vert P \Vert_{\mathcal P(^2 \ell_2^3)}^{1-\theta} \Vert P \Vert_{\mathcal P(^2 \ell_\infty^3)}^\theta =2^{1-\theta} 5^\theta = 2^{\frac{2}{p}}5^{1-\frac{2}{p}}.$$
Thus,
$\ccc(2,\ell_p^3(\mathbb C)) \geq (\frac{6}{5})^{1-\frac{2}{p}}$,
which concludes the proof.
\end{proof}

It should be noted that we cannot have a statement as in the previous lemma for dimension 2. For example,  $\ccc(2,\ell_\infty^2(\mathbb C))=1$ (see \cite[Proposition 3]{S2}).

Given a number $1 \leq r \leq \infty$, we denote by $r'$ its conjugate exponent (i.e., $1/r + 1/r'=1$).
We now derive Theorem \ref{teo tipo y cotipo}. 

\begin{proof}[Proof of Theorem \ref{teo tipo y cotipo}]

We follow a similar procedure used in \cite[Section 4]{bayart2010weak}. For any $\varepsilon >0$ and $d\in \zN$, $X$ admits two quotients: $X/Y_p$ which is $(1+\varepsilon)$-isomorphic to $\ell_{\ppp(X^*)'}^d$ and $X/Y_q$ which is $(1+\varepsilon)$-isomorphic to $\ell_{\qqq(X^*)'}^d$ (see \cite[Proposition 2.3]{bayart2010weak}).

If $\ppp(X^*)\neq 2$ then $\ppp(X^*)'\neq 2$. Therefore, we have that $\ccc(2,\ell_{\ppp(X^*)'}^3)>1$ by the previous lemma. Thus, there is $\varepsilon >0$ such that
$$\frac{\ccc(2,\ell_{\ppp(X^*)'}^3) }{(1+\varepsilon)^2} >1.$$
Then, using \cite[Lemma 0]{S1}, we obtain
$$\ccc(2,X)\geq \ccc(2,X/Y_p) \geq \frac{\ccc(2,\ell_{\ppp(X^*)'}^3)}{(1+\varepsilon)^2}>1,$$
which is a contradiction. If $\qqq(X^*)\neq 2$, the same argument but using the fact that $\ccc(2,\ell_{\qqq(X^*)'}^2)>1$, yields again a contraction.

All this shows that $\ppp(X^*)=\qqq(X^*)=2$. Taking into account that $\ppp(X^*)>1$, we get that
$$\ppp(X)=\ppp(X^{**})=\qqq(X^*)'=2\,\,\,\mbox{ and }\,\,\,\qqq(X)=\qqq(X^{**})\le \ppp(X^*)'=2,$$
which concludes the proof.

\end{proof}

It should be noted that the reciprocal of Theorem \ref{teo tipo y cotipo} is not valid. 
 For  example, if $\mathcal{H}$ is a Hilbert space then $\mathcal{H} \oplus_{2} \ell_1^2$ is isomorphic to a Hilbert space (an so $\ppp(\mathcal{H} \oplus_{2} \ell_1^2)=\qqq(\mathcal{H} \oplus_{2} \ell_1^2)=2$), but clearly $\ccc(2,\mathcal{H} \oplus_{2} \ell_1^2) \geq \ccc(2, \ell_1^2) =2$.
This is not at all surprising since we cannot characterize an isometric property (such as the \textsl{symmetric operator norm property}) with an isomorphic property (like $\ppp(X)=\qqq(X)=2$).



\section{Nuclear norm of the product of polynomials} \label{nuclear}

A classical result of Gupta establishes a duality between $\mathcal P_N(^kX)$ and $\mathcal P(^kX^*)$, whenever $X^*$ has the approximation property \cite[Proposition 2.10]{dineen1999complex}. Therefore it is natural that certain constants related to polarization of polynomials in $\mathcal P(^kX^*)$ have their counterpart in  $\mathcal P_N(^kX)$.
As mentioned in the introduction our aim in this section is to study $\mmm(k_1,\ldots,k_n,X)$, the best constant  such that inequality \eqref{problema} holds. The problem of estimating/bounding the constant $\mmm$ was previously considered, for instance, in \cite[Lemma 15]{dineen1971holomorphy}, \cite[Corollary 2]{abuabara1979version}, \cite{nicodemi1981homomorphisms} and \cite[Exer. 2.63]{dineen1999complex}.

If $k_1,\ldots, k_n$ are natural numbers such that $k_1+\cdots+k_n=k$ and $\mathbf{x}_1,\ldots, \mathbf{x}_n\in X$, we denote by $(\mathbf{x}_1^{k_1},\ldots, \mathbf{x}_n^{k_n})$,$k$-tuple where $x_j$ appears $k_j$-times. 
In many situations the symmetric $k$-linear form is evaluated in this kind of $k$-tuples (which admit many repetitions). 
For example, if $\hat d ^j P $ stands for the $j$th-derivative of $P$ (see \cite[Chapter 3]{dineen1999complex}), we have $\frac{\hat d ^j P (\mathbf{x}_1)}{j!}(\mathbf{x}_2)=\binom{k}{j} \overset{\vee}{P}(\mathbf{x}_1^{k-j},\mathbf{x}_2^{j})$.
 An inequality by Harris \cite{H} states that if $\mathbf{x}_1,\ldots, \mathbf{x}_n$ are vectors in $B_X$, the unit ball of a complex space $X$, then for every $P \in \mathcal P(^kX)$
\begin{equation}\label{ec polarizacion harris}
|\overset{\vee}{P} (\mathbf{x}_1^{k_1},\ldots, \mathbf{x}_n^{k_n})|\leq \frac{k_1!\cdots k_n!}{k!}\frac{k^k}{k_1^{k_1}\cdots k_n^{k_n}} \Vert P\Vert.
\end{equation}
It is worthwhile to mention that in general this bound cannot be reduced. For simplicity, it is natural to define for a fixed space $X$ and $k_1+\dots+k_n=k$,
the norm
\begin{equation}
    \Vert P \Vert_{k_1,\dots,k_n;X}:= \sup_{x_1, \dots, x_n \in B_X} |\overset{\vee}{P} (\mathbf{x}_1^{k_1},\ldots, \mathbf{x}_n^{k_n})|.
\end{equation}

We denote the best constant $C>0$ such that 
\begin{equation}\label{ec polarizacion harris2}
\Vert P \Vert_{k_1,\dots,k_n;X}\leq C \Vert P\Vert,
\end{equation}
for every $P \in \mathcal P(^kX)$  as $\ccc(k_1, \dots, k_n;X)$.
These problems have been studied by several authors. We refer the reader to the works cited above, the articles \cite{carando2019symmetric, H, kim2014polarization,  papadiamantis2016polynomial, S1, S2, T}  and the references therein for more information and results on this topic.

\begin{thm}\label{teo producto nucleares} Let $X$ be a Banach space such that $X^*$ has the approximation property, then
$$\mmm(k_1,\ldots,k_n,X)=\ccc(k_1,\ldots,k_n,X^*).$$
\end{thm}

For the proof we will use basic theory of symmetric tensor product of Banach spaces. For an introduction to this topic and the notation we use next,  we refer the reader to Floret's survey \cite{floret1997natural}.

\begin{proof}
First let us see that $\mmm(k_1,\ldots,k_n,X)\leq\ccc(k_1,\ldots,k_n,X^*)$. By the definition of the nuclear norm, it is enough to prove that
\begin{equation}\label{acotacion nuclear}
 \Vert \varphi_1^{k_1} \cdots \varphi_n^{k_n}\Vert_{\mathcal P_{N}(^{k}X)}\leq  \ccc(k_1,\ldots,k_n,X^*),  
\end{equation}
for norm one functionals $\varphi_1, \dots, \varphi_n \in X^*$.

Let 
$$J:\otimes^{{k},s}_{\pi_s} X^* \twoheadrightarrow \mathcal{P}_{N}(^{k}X)$$
be the natural metric surjection (see \cite[Section 2]{floret1997natural}). It is not hard to see that
$$J(\sigma[(\otimes^{k_1} \varphi_1)\otimes \cdots\otimes (\otimes^{k_n} \varphi_n )])= \varphi_1^{k_1}\cdots  \varphi_n^{k_n},$$
where $\sigma:\otimes^{k}_{\pi} E^* \rightarrow \otimes^{k,s}_{\pi_s} E^*$ is the symmetrization operator. Therefore
$$\Vert \varphi_1^{k_1} \cdots \varphi_n^{k_n} \Vert_{P_{N}(^{k}X)}\leq \pi_s(\sigma[(\otimes^{k_1} \varphi_1)\otimes \cdots\otimes (\otimes^{k_n} \varphi_n )]).$$
By duality, the projective symmetric norm is computed as follows
\begin{eqnarray*}
\pi_s(\sigma[(\otimes^{k_1} \varphi_1)\otimes \cdots\otimes (\otimes^{k_n} \varphi_n )]) &=& \sup \{|Q(\sigma((\otimes^{k_1} \varphi_1)\otimes \cdots\otimes (\otimes^{k_n} \varphi_n ))|\} ,\
\end{eqnarray*}
where the supremum is taken over all the norm one polynomials in $\mathcal{P}(^{k}X^{*})$. For any such $Q$ we have
\begin{eqnarray*}
|Q(\sigma((\otimes^{k_1} \varphi_1)\otimes \cdots\otimes (\otimes^{k_n} \varphi_n ))|&=& |\overset{\vee}{Q}(\sigma((\otimes^{k_1} \varphi_1)\otimes \cdots\otimes (\otimes^{k_n} \varphi_n ))| \\
&=&|\overset{\vee}{Q}(\varphi_1^{k_1},\ldots, \varphi_n^{k_n} ))|\\
&\leq & \ccc(k_1,\ldots,k_n,X^*), \
\end{eqnarray*}
which implies \eqref{acotacion nuclear}.

Now let us prove that $\ccc(k_1,\ldots,k_n,X^*)\leq \mmm(k_1,\ldots,k_n,X)$. Given $\varepsilon >0$, take $\varphi_1,\ldots, \varphi_n$ norm one vectors in $X^*$ and a norm one polynomial $Q\in  \mathcal{P}(^{k}X^{*})$ such that
$$\frac{\ccc(k_1,\ldots,k_n,X^*) }{(1+\varepsilon)} <|\overset{\vee}{Q}(\varphi_1^{k_1},\ldots, \varphi_n^{k_n} )|.$$
By the computations done above we have that
\begin{eqnarray*}
|\overset{\vee}{Q}(\varphi_1^{k_1},\ldots, \varphi_n^{k_n} ))| &\leq& \pi_s(\sigma[(\otimes^{k_1} \varphi_1)\otimes \cdots\otimes (\otimes^{k_n} \varphi_n )]). \\
&=&  \Vert \varphi_1^{k_1} \cdots \varphi_n^{k_n} \Vert_{\mathcal P_{N}(^{k}X)}, \
\end{eqnarray*}
where the last equality is valid due to the approximation property of $X^*$.
Putting all together we get
$$\frac{\ccc(k_1,\ldots,k_n,X^*) }{(1+\varepsilon)}< \Vert  \varphi_1^{k_1} \cdots \varphi_n^{k_n} \Vert_{\mathcal P_{N}(^{k}X)}.$$
Recall that $\Vert \varphi_i^{k_i}\Vert_{\mathcal P_{N}(^{k_i}X)} = 1$ for every $i=1, \dots,n$. Then,
$$\frac{\ccc(k_1,\ldots,k_n,X^*) }{(1+\varepsilon)} <  \mmm(k_1,\ldots,k_n,X),$$
which concludes the proof.
\end{proof}

\begin{rem}
Note that for proving the inequality $\mmm(k_1,\ldots,k_n,X)\leq\ccc(k_1,\ldots,k_n,X^*)$ the hypothesis about the  approximation property is unnecessary.
\end{rem}

\begin{cor}\label{cor espacio lp} For $L_p$-spaces we have
$$\mmm(k_1,\ldots,k_n,L_p(\mu))=\ccc(k_1,\ldots,k_n,L_{p^\prime}(\mu)),$$
where $\frac{1}{p}+\frac{1}{p^\prime}=1.$
\end{cor}
\begin{proof}
If $p\neq \infty$ this is a particular case of the above theorem since $L_p(\mu)^*=L_{p^\prime}(\mu)$ has the approximation property. Thus we only need to prove the result for an infinite dimensional space $L_\infty(\mu)$. 

By Theorem \ref{teo producto nucleares} and \cite[Theorem 1]{H} we know that
\begin{eqnarray*}
\mmm(k_1,\ldots,k_n,L_\infty(\mu))&=&\ccc(k_1,\ldots,k_n,L_{\infty}(\mu)^*) \\
&\leq& \frac{k^k}{k_1^{k_1}\cdots k_n^{k_n}} \frac{k_1!\cdots k_n!}{k!} \\
&=&\ccc(k_1,\ldots,k_n,L_1(\mu)).\
\end{eqnarray*}

On the other hand, since every continuous $k$-homogeneous polynomial on $L_1(\mu)$ extends to a $k$-homogeneous polynomial of the same norm defined on the bidual $L_1(\mu)^{**}$ (see \cite{aron1978hahn}), we have the other inequality
\begin{eqnarray*}
\ccc(k_1,\ldots,k_n,L_1(\mu))&\leq& \ccc(k_1,\ldots,k_n,L_1(\mu)^{**})\\
&=&\ccc(k_1,\ldots,k_n,L_{\infty}(\mu)^*)\\
&=&\mmm(k_1,\ldots,k_n,L_\infty(\mu)),\
\end{eqnarray*}
and this concludes the proof.
\end{proof}


We continue with some comments on the constant $\mmm$ for some classical spaces. In \cite{S1, S2} Sarantopoulos studied the polarization constants for several spaces. Using Theorem \ref{teo producto nucleares} and some of Sarantopoulos' results we derive the following
\begin{prop}  Let $k_1,\ldots, k_n$ be natural numbers such that $k_1+\cdots+k_n=k$.
\begin{enumerate}
    \item If $k\leq p $,  then for any complex space $L_p(\mu)$ we have
    \begin{equation}\nonumber
        \mmm(k_1,\ldots, k_n;  L_{p}(\mu))\leq \left(\frac{k^k}{k_1^{k_1}\cdots k_n^{k_n}}\right)^{1-\frac{1}{p}} \frac{k_1!\cdots k_n!}{k!}.
    \end{equation}
    Moreover, this inequality is in fact an equality provided that the dimension of the space $ L_{p}(\mu)$ is at least $n$.
    \item If $k$ is an even number,  $\frac{p}{p-1} \leq \frac{k}{2}$, then for any complex space $L_p(\mu)$ we have
    $$\mmm\left(\frac{k}{2},\frac{k}{2},  L_p(\mu)\right)=1.$$
    \item For  any complex Hilbert space $\mathcal{H}$
    $$\mmm(k_1,\ldots, k_n;  \mathcal{H} \oplus_1 \zC)=1.$$
\end{enumerate}
\end{prop}

Here we used only some of the results proved in \cite{S1, S2}. Furthermore, in the literature there are several other works in which the polarization constants, or related inequalities, are studied. See for example the aforementioned articles \cite{carando2019symmetric, H, kim2014polarization, papadiamantis2016polynomial, T}.

Note that over the range $1 \leq p \leq k^\prime$, where $\frac{1}{k}+\frac{1}{k^\prime}=1$, the values of the polarization constants of infinite dimensional complex $L_p$-spaces are known. We now give some more information over the range $k^\prime \leq p \leq k$. The following result can be derived  by interpolating  the operators \mbox{$T_0:\ell_2^n(L_2(u))\rightarrow L_2(L_2(u), t)$} and $T_1:\ell_1^n(L_q(u))\rightarrow L_\infty(L_q(u), t)$  both defined as
$$(f_1(u),\ldots, f_n(u))\rightarrow f_1(u) s_1(t) + \cdots +f_n(u) s_n(t)$$
and mimicking the proof of \cite[Theorem 1]{S1}.

\begin{prop} If $k^\prime \leq p \leq k$, then for any complex  space $L_p(\mu)$ we have
\begin{equation}\label{ec other range}
    \ccc(k_1,\ldots, k_n; L_p(\mu)) = \mmm(k_1,\ldots, k_n; L_{p'}(\mu)) \leq \left(\frac{k^k}{k_1^{k_1}\cdots k_n^{k_n}}\right)^{\frac{1}{k^\prime}} \frac{k_1!\cdots k_n!}{k!}.
\end{equation} \end{prop}

In particular, in the range of interest, the constants can not be bigger than on $[1,k']$.  This bound is an improvement from the one given in \cite[Proposition 6]{S2}. Although in general this is not an optimal bound, in the case $p=k^\prime$  we recover the constant in  \cite[Theorem 1]{S1} which is optimal. Moreover, Proposition \ref{prop cont} below implies that the bound given in \eqref{ec other range} is 
arbitrarily close to the actual constant, provided that $p$ is close enough to $k^\prime$.

Although the exact value of the polarization constants is not known for every $L_p$-space, Sarantopoulos proved that for a fixed value of $k$, $\ccc(k, L_p(\mu))$ is an increasing function of $p$ over the range $2\leq p \leq \infty$. The same holds true --with identical arguments-- for $\ccc(k_1,\ldots, k_n;  L_p(\mu))$. 
Of course, by Corollary \ref{cor espacio lp}, all these statements have their  counterpart for $\mmm(k_1, \dots, k_n; L_p(\mu))$.
The following proposition shows that the constants  $\ccc$ and $\mmm$ over $L_p$-spaces are continuous on the parameter $1 \leq p \leq \infty$. 

\begin{prop}\label{prop cont}
Let $k_1, \dots, k_n$ be natural numbers. Then, the constants  $\ccc(k_1, \dots, k_n; L_p(\mu))$ and $\mmm(k_1, \dots, k_n; L_p(\mu))$ are continuous functions on  $1 \leq p \leq \infty.$
\end{prop}

\begin{proof} By Corollary \ref{cor espacio lp} we only need to prove the continuity of $\ccc(k_1, \dots, k_n; L_p(\mu))$.  Let us assume first that we are dealing with infinite  dimensional $L_p$ spaces. Given $\epsilon>0$ we will see that 
\begin{equation} \label{desig continuidad}
    \ccc(k_1, \dots, k_n; L_p(\mu)) \leq \ccc(k_1, \dots, k_n; L_q(\mu)) (1 + \epsilon),
\end{equation}
provided that $\vert p -q \vert$ is small enough. 

Denote by $k:=k_1+ \dots+k_n$ and let $\eta>0$ fixed (to be defined later). Given $P\in \mathcal P(^kL_p(\mu))$, consider $x_1, \dots, x_n \in B_{L_p(\mu)}$ such that 
$$(1-\eta) \Vert P \Vert_{k_1,\dots,k_n;L_p(\mu)} \leq  |\overset{\vee}{P} (\mathbf{x}_1^{k_1},\ldots, \mathbf{x}_n^{k_n})|.$$
 By  \cite[Theorem A]{pelczynski1974localization}, we know there is a natural number $M:=M(n,\eta)$ with the following property: given a subspace $E$ of $L_p$ of dimension less than or equal to $n$, there is a subspace $F \supset E$ of dimension $m \leq M$ such that $F$ is $(1+\eta)$-complemented $(1+\eta)$-isomorph of $\ell_p^m$.
 We will use this result for $E:=\mbox{span}\{x_1, \dots, x_n\}.$

 We denote by $\iota_F: F \to L_p(\mu)$ the canonical inclusion.
 Let $T:F \to \ell_p^m$ be an isomorphism such that $\Vert T\Vert \Vert T^{-1}\Vert  \leq 1+\eta$. Given $1 \leq q \leq \infty$ consider $S:\ell_p^m \to \ell_q^m$ so that $ d(\ell_p^m,\ell_q^m)=\Vert S\Vert \Vert S^{-1}\Vert ,$ where $d$ stands for the Banach-Mazur distance. 
 We have the following inequalities
 \begin{align*}
   (1-\eta) \Vert P \Vert_{k_1,\dots,k_n;L_p(\mu)} & \leq  |\overset{\vee}{P} (\mathbf{x}_1^{k_1},\ldots, \mathbf{x}_n^{k_n})|\leq \Vert P \circ \iota_F \Vert_{k_1,\dots,k_n;F} \\
   & = \Vert P \iota_FT^{-1}S^{-1}ST \Vert_{k_1,\dots,k_n;F} \leq \Vert P \iota_FT^{-1}S^{-1} \Vert_{k_1,\dots,k_n;\ell_q^m} \Vert S \Vert^k \Vert T \Vert^k \\
   & \leq \ccc(k_1, \dots, k_n; \ell_q^m) \Vert P \iota_FT^{-1}S^{-1} \Vert_{\mathcal P (^k\ell_q^m)} \Vert S \Vert^k \Vert T \Vert^k \\ &\leq \ccc(k_1, \dots, k_n; \ell_q^m) \Vert P \iota_F \Vert_{\mathcal P (^kF)} \Vert S \Vert^k \Vert S^{-1} \Vert^k \Vert T \Vert^k  \Vert T^{-1} \Vert^k \\
   &\leq \ccc(k_1, \dots, k_n; \ell_q^m)    d(\ell_p^m,\ell_q^m)^k (1+\eta)^k \Vert P \Vert_{\mathcal P (^kL_p(\mu))} \\
   &\leq \ccc(k_1, \dots, k_n; L_q(\mu))    d(\ell_p^m,\ell_q^m)^k (1+\eta)^k \Vert P \Vert_{\mathcal P (^kL_p(\mu))},
 \end{align*}
where the last inequality is due to \cite[Ecuation (2)]{S2}. Since this holds for any polynomial $P$, by the mere definition of the constant  $\ccc(k_1, \dots, k_n; L_p(\mu))$, we obtain
 \begin{equation}
     \ccc(k_1, \dots, k_n; L_p(\mu)) \leq \ccc(k_1, \dots, k_n; L_q(\mu)) d(\ell_p^m,\ell_q^m)^k (1+\eta)^k (1- \eta)^{-1}.
 \end{equation}

 If we pick beforehand $\eta>0$ such that $(1+\eta)^k (1- \eta)^{-1} \leq (1+\varepsilon)^{1/2}$ and if $q$ is close to $p$ in order that $1 \leq p,q \leq 2$ or $2 \leq p,q \leq \infty$ and $M^{\vert \frac{1}{p} - \frac{1}{q} \vert k} \leq
 (1+\varepsilon)^{1/2}$ we have, by \cite[Proposition 37.6 (i)]{tomczak1989banach}, that $d(\ell_p^m,\ell_q^m)^k =  m^{\vert \frac{1}{p} - \frac{1}{q} \vert k} \leq
 (1+\varepsilon)^{1/2}$. We therefore obtain \eqref{desig continuidad}, which concludes the proof for the infinite dimensional case.

The finite dimensional case, $\ell_p^N$, follows directly from the fact that $d(\ell_p^N,\ell_q^N)=N^{\vert \frac{1}{p} - \frac{1}{q} \vert}$ if $p$ and $q$ are close enough (\cite[Proposition 37.6 (i)]{tomczak1989banach}).
 \end{proof}

Using Bolzano's theorem we obtain that the constant $\ccc(k, \cdot)$ can attain any value between $1$ and $k^k/k!$.

\begin{cor}
Given $1 \leq c \leq k^k/k!$, there is $1 \leq q \leq 2$ such that $\ccc(k,\ell_q)=c.$
\end{cor}

\subsection*{Acknowledgments}
We thank the referees for their valuable comments that allow us to improve the presentation of the manuscript, as well as for the  alternative proof of Proposition \ref{ele-1}.

We thank S. Drury for his kind answers and the encouragement given at the beginning of this project.



\end{document}